\documentclass[12pt,a4paper]{article}
\usepackage[utf8]{inputenc}
\usepackage{amsmath}
\usepackage{amsfonts}
\usepackage{amssymb}
\author{Mark Shusterman}
\title{Free subgroups of finitely \\generated free profinite groups}
\usepackage{mathtools}
\usepackage{graphicx}
\usepackage{hyperref}
\usepackage{tikz}
\usetikzlibrary{matrix,arrows,decorations.pathmorphing}
\newtheorem{theorem}{Theorem}[section]
\newtheorem{lemma}[theorem]{Lemma}
\newtheorem{proposition}[theorem]{Proposition}
\newtheorem{corollary}[theorem]{Corollary}
\usepackage{hyperref}
\usepackage[super]{nth}

\newenvironment{proof}[1][Proof]{\begin{trivlist}
\item[\hskip \labelsep {\bfseries #1}]}{\end{trivlist}}
\newenvironment{definition}[1][Definition]{\begin{trivlist}
\item[\hskip \labelsep {\bfseries #1}]}{\end{trivlist}}

\newcommand{\lemref}[1]{\hyperref[#1]{Lemma \ref*{#1}}}
\newcommand{\thmref}[1]{\hyperref[#1]{Theorem \ref*{#1}}}
\newcommand{\propref}[1]{\hyperref[#1]{Proposition \ref*{#1}}}
\newcommand{\corref}[1]{\hyperref[#1]{Corollary \ref*{#1}}}
\newcommand{\defref}[1]{\hyperref[#1]{Definition \ref*{#1}}}
\newcommand{\rank}{\text{rank}}

\begin{document}

\maketitle

\abstract{We give new and improved results on the freeness of subgroups of free profinite groups: A subgroup containing the normal closure of a finite word in the elements of a basis is free; Every infinite index subgroup of a finitely generated nonabelian free profinite group is contained in an infinitely generated free profinite subgroup. These results are combined with the twisted wreath product approach of Haran, an observation on the action of compact groups, and a rank counting argument to prove a conjecture of Bary-Soroker, Fehm, and Wiese, thus providing a quite general sufficient condition for subgroups to be free profinite. As a result of our work, we are able to address a conjecture of Jarden on the Hilbertianity of fields generated by torsion points of abelian varieties.}

\section{Introduction}

One of the most celebrated theorems of classical group theory is the theorem of Nielsen and Schreier stating that every subgroup of a free group is free. This theorem has attracted much attention, and proofs from rather diverse mathematical disciplines like abstract group theory, algebraic topology, geometric group theory, and Bass-Serre theory have been given (see \cite{DHLS}, \cite[Proposition 17.5.6]{FJ}, \mbox{\cite[Chapter 14]{Hig}}, \cite{How}, \cite[pp. 383-387]{Rot}, \cite{RS}, and \cite[Chapter 2]{SW}).

In Field Arithmetic, a central task is to extend the Nielsen-Schreier theorem to the realm of profinite groups. Unfortunately, the $2$-Sylow subgroup of a free profinite group has only $2$-groups as finite images, so it is not a free profinite group. It is therefore apparent that the theorem does not extend to the profinite setting in the strict sense, and much of the work carried out by  Bary-Soroker,  Binz, Chatzidakis, Fehm, Gildenhuys, Haran, Harbater, Iwasawa, Jarden, Lubotzky, Lim, Melnikov, Neukirch, Ribes, Steinberg, V.D. Dries, Wenzel, Wiese, Zalesskii and others, produced sufficient conditions for the permanence of profinite freeness in closed subgroups.

Besides the intrinsic group theoretic interest of exploring a free profinite group by examining the structure of its closed subgroups (these are the projective objects in the category of profinite groups), extensions of the Nielsen-Schreier theorem to profinite groups prove to be of arithmetical importance for two reasons:

\begin{itemize}

\item Free profinite groups arise as absolute Galois groups of fields, so any information about a closed subgroup, tells us something about the absolute Galois group of a field extension. In particular, the profinite freeness of a closed subgroup provides us with a solution of the inverse Galois problem over its fixed field in a rather strong sense.

\item Free profinite groups and absolute Galois groups of Hilbertian fields (see \defref{Hildef}) exhibit many similar properties, and the techniques used to study these object have a lot in common (for example, \cite{Ja80}, \cite{Ja97}, and \cite{JL99}). Most notably, there is an analogy between sufficient conditions for the profinite freeness of a subgroup, and sufficient conditions for Hilbertianity of a field extension.

\end{itemize}

This analogy is called the Twinning Principle of Jarden and Lubotzky (see \cite{JL92} and \cite{Ha}, \cite{BP} for later work in this direction). This principle suggested the validity of several theorems, fruitfully passing from free profinite groups to Hilbertian fields and vice versa (see \cite{BS}, \cite{Ha}, and \cite{Ha99}). It is therefore of no surprise that Bary-Soroker, Fehm, and Wiese conjectured a generalization of \cite[Theorem 5.7]{BFW} (the countably infinite case of \thmref{FSLThm}) once they have established the field theoretic analogue in \cite[Theorem 3.2]{BFW}. In our work we prove this generalization, thus giving a quite general sufficient condition for the profinite freeness of a closed subgroup of a free profinite group.

\begin{theorem} \label{FSLThm}

Let $F$ be a nonabelian free profinite group, and let $N \lhd_c F$ be a closed normal subgroup such that $F/N$ is of finite abelian-simple length. Then any closed subgroup \mbox{$N \leq H \leq_c F$} is a free profinite group.

\end{theorem}

The definition of the class of groups of fixed abelian-simple length appears in \cite[Section 2]{BFW}, and is given in the section devoted to these groups in our work (Section 5, \defref{FSLDef}). The reason to consider this class of groups is twofold: First, this class, being defined by restricting the length of certain normal series, contains several natural classes of profinite groups like the class of abelian groups, and the class of groups with a vanishing $k$-th commutator for some fixed $k \in \mathbb{N}$ (see the abelian case in \cite[Corollary 25.4.8]{FJ} and the generalization of it to the solvable one in \corref{SolveCor}). Second, groups from this class appear as certain images of Galois groups of field extensions formed by such arithmetical operations as adjoining torsion points from abelian varieties (or alternatively, taking fields fixed by kernels of modular Galois representations). For more on this point, see the proof of \corref{AbVarCor} and that of \thmref{AdicRepsThm}.

However straightforward is the analogy between \thmref{FSLThm} and its counterpart in fields (\cite[Theorem 3.2]{BFW}), the proof of \cite[Theorem 3.2]{BFW} does not provide the full arsenal of tools necessary to prove \thmref{FSLThm}. The most challenging case of \thmref{FSLThm} is the finitely generated one. Some of the machinery required for its proof is therefore, combinatorial and is developed in the $\nth{3}$ section, in which we prove the verbal subgroup theorem generalizing a result of Melnikov (he was able to prove the case $H = N$ of \thmref{FreeVerbalSubgroupThm}, see \cite[Theorem 8.9.1]{RZ} or \cite[Proposition 25.8.1]{FJ}) in the case of a finitely generated free group:

\begin{theorem} \label{FreeVerbalSubgroupThm}

Let $F$ be a finitely generated free profinite group, and let $H \leq_c F$ be a closed subgroup of infinite index in $F$. Suppose that there is some closed normal subgroup $N \lhd_c F$ contained in $H$, such that $N$ contains a nontrivial finite word in the elements (and inverses) of some basis for $F$. Then $H$ is a free profinite group of rank $\aleph_0$.

\end{theorem}     

We feel that the verbal subgroup theorem is a useful tool in establishing the profinite freeness of closed subgroups, which will be utilized and generalized in future work. Furthermore, we prove the intermediate subgroup theorem which resembles another theorem of Melnikov (see \cite[Theorem 8.9.9]{RZ} or \cite[Proposition 25.8.3]{FJ} for his analogue for normal subgroups):

\begin{theorem} \label{FreeInterSubThm}

Let $F$ be a nonabelian finitely generated free profinite group, and let $H \leq_c F$ be a closed subgroup of infinite index in $F$. Then there exists a free profinite subgroup \mbox{$H \leq L \leq_c F$} of rank $\aleph_0$.

\end{theorem}

From the very statement of this theorem it is apparent that it aims at creating a link between profinite freeness in the finitely generated case, and profinite freeness in the infinitely generated one. We hope to further generalize this result, and to combine it with the fact that a countably generated free profinite group is an inverse limit of finitely generated ones, to establish a connection between the following three objects using the twinning principle: 

\begin{itemize}

\item The absolute Galois group of a Hilbertian field (an arithmetic object).

\item The free profinite group of rank $\aleph_0$ (a group-theoretic object).

\item A finitely generated free profinite group (a combinatorial object). 

\end{itemize}

It should be noted that making precise the analogy between these objects served as the initial motivation for trying to prove \thmref{FSLThm}, in a hope that the proof of it will contain some of the arguments needed for a "Twinning Theorem" that will shed light on the mechanism behind the twinning principle. An example of such a theorem is the result stating that a countable pseudo algebraically closed field (see \defref{Hildef}) is Hilbertian if and only if its absolute Galois group is free profinite of infinite rank.

In the $\nth{4}$ section, we unveil the last ingredient of the proof of \thmref{FSLThm} - the decomposition of the action of a compact group on a product of nonabelian finite simple groups. In the $\nth{5}$ section, after giving some definitions and basic properties, we prove the infinitely generated case of \thmref{FSLThm}, and handle the finitely generated case by combining a rank counting argument with the tools developed throughout.

In the last section, we mainly give applications of \thmref{FSLThm}. Most notably, we prove \thmref{AdicRepsThm} and use it to deduce the Hilbertianity of fields coming from the Kuykian conjecture (see \cite[Theorem 1.3]{BFW}, \cite{FP}, \cite{FJP}, and \cite{Ja10}):

\begin{theorem} \label{AlmostThm}

Let $k \geq 2$. For almost every $\sigma_1, \dots, \sigma_k \in \mathrm{Gal}(\overline{\mathbb{Q}}/\mathbb{Q})$ with respect to the Haar measure, the fixed field of $\sigma_1, \dots, \sigma_k$ inside $\overline{\mathbb{Q}}$ denoted by $L$, satisfies the following: 

For any abelian variety $A/L$, every infinite extension $K$ of $L$ which is contained in $L(A_{\mathrm{tor}})$, is Hilbertian.

\end{theorem}

\section*{Acknowledgments} I would like to sincerely thank Lior Bary-Soroker for suggesting his conjecture as an M.Sc. project, and for all his help as a supervisor. Special thanks go to Moshe Jarden and Dan Haran for their careful reading of the drafts, and various tips. I would also like to thank Yves de Cornulier for helping me with the proof of \lemref{CompactLem}. This research was partially supported by a grant of the Israel Science Foundation with cooperation of UGC no. 40/14.

\section{Background}

Here we survey some basic notation and definitions essential for our work, and prove or cite some basic results.

Our working category will be that of profinite groups. In view of that, group theoretic properties should be understood in the topological sense (unless stated otherwise). For example, subgroups are closed (denoted by $\leq_c$, or $\leq_o$ if the subgroup is also open), homomorphisms are continuous, and so on.

Let $G$ be a profinite group and let $X$ be a subset of G. We say that $X$ \textit{\textbf{generates}} $G$ if the abstract subgroup of $G$ generated by $X$ is dense in $G$, and that $X$ \textbf{\textit{converges to 1}} if every open subgroup $U$ of $G$ contains all but a finite number of the elements of $X$. $G$ is called \textbf{\textit{finitely generated}} if it contains a finite subset $X$ that generates it in the aforementioned topological sense. More generally, we define $d(G)$ to be the smallest cardinality of a set of generators of $G$ converging to $1$. 

By an \textbf{\textit{embedding problem}} for a profinite group $G$ we mean a diagram $$\begin{tikzpicture}[scale=1.5]
\node (B) at (1,1) {$G$};
\node (C) at (0,0) {$A$};
\node (D) at (1,0) {$B$};
\path[->,font=\scriptsize,>=angle 90]
(B) edge node[right]{$\theta$} (D)
(C) edge node[above]{$\alpha$} (D);
\end{tikzpicture}$$ where $\theta, \alpha$ are continuous epimorphisms of profinite groups. The problem is called \textbf{\textit{finite}} if $A$ is a finite group. We say that the problem is \textbf{\textit{solvable}} if there exists a continuous homomorphism $\lambda \colon G \rightarrow A$ making the diagram commutative, and \textbf{\textit{properly solvable}} if $\lambda$ can be taken to be surjective.

Let $F$ be a profinite group generated by $X \subseteq F$ converging to $1$. We say that $F$ is a \textbf{\textit{free profinite group}} if every function $\alpha \colon X \rightarrow L$ to a finite group $L$ with $\alpha(x) \neq 1$ for only finitely many $x \in X$, extends to a continuous group homomrphism $\beta \colon F \rightarrow L$. In this case, $X$ is said to be a \textbf{\textit{profinite basis}} of $F$, and the \textbf{\textit{rank}} of $F$ is defined by $\rank(F) = |X|$. The rank does not depend on the choice of a basis, and is given intrinsically by $d(F)$ (determining $F$ up to isomorphism). Analogously, the rank of an abstract free group is the size of a basis (a free generating set - a set which generates the group and satisfies no relations).

Let $F$ be a free profinite group on a set $X$ (i.e $X$ is a profinite basis). Let $\Phi$ be the abstract subgroup generated by $X$ ($\Phi$ is free by \cite[Corollary 3.3.14]{RZ} and \cite[Proposition 3.3.15]{RZ} applied to the formation of all finite groups). By \cite[Proposition 5.1.3']{Wil}, $F$ is isomorphic to the profinite completion of $\Phi$ with respect to the family of those finite-index normal subgroups (of $\Phi$) that contain all but finitely many elements of $X$ (this allows us to construct the free profinite groups as completions of abstract free groups). Note that if $X$ is finite, (we will be mostly interested in this case) then the free profinite group on $X$ is just the ordinary profinite completion of the abstract free group on $X$. In the following claim, we generalize and simplify the argument given in \cite[Theorem 3.6.2, Case 1]{RZ}.

\begin{proposition} \label{AbstProfBasisProp}

Let $F$ be a free profinite group on a finite set $X$, and let $U \leq_o F$. Denote by $\Phi$ the free abstract subgroup of $F$ generated by $X$, and let $Y$ be a basis for $U \cap \Phi$. Then $Y$ is a profinite basis for $U$.

\end{proposition}

\begin{proof}

Set $U_0 = U \cap \Phi$, and note that since $\Phi$ is dense in $F$, $U_0$ is dense in $U$ (that is $\overline{U_0} = U$ because the intersection of a dense set with a nonempty open set is dense). Therefore, $Y$ generates $U$, so in order to show that $Y$ is a profinite basis, we take some mapping to a finite group $\alpha \colon Y \rightarrow L$. Since $Y$ is a basis of $U_0$, there is a homomorphism $\gamma \colon U_0 \rightarrow L$ extending $\alpha$. Set $N = \text{Ker}(\gamma)$, and note that by \cite[Proposition 3.2.2 (d)]{RZ} (applied to $\Phi$, and the formation of all finite groups), $\bar{\gamma} \colon U \rightarrow U/\bar{N} \cong U_0/N \cong L$ extends $\gamma$, so we have successfully extended $\alpha$ to $U$. $\blacksquare$  

\end{proof}

It is important to note that \propref{AbstProfBasisProp} does not generalize to the case of an infinite basis, since in this case $Y$ does not have to converge to $1$.

For an abstract free group $F$ on a set $X$, each $w \in F$ has a unique reduced presentation as a finite product of elements of $X$ and their inverses: $$ w = \prod_{i = 1}^{m} x_i^{\epsilon_i} , m \in \mathbb{N}$$ with $x_i \in X$, $\epsilon_i \in \{\pm 1\}$, and $x_i^{\epsilon_i}x_{i+1}^{\epsilon_{i + 1}} \neq 1$ for all $1 \leq i \leq m$. We set $\text{length}_X(w) = m$, and say that $w$ starts with $x_1^{\epsilon_1}$. The following proposition guarantees the existence of bases with additional properties.

\begin{proposition} \label{BasesProp}

Let $F$ be a free abstract (respectively, profinite) group with a basis (respectively, profinite basis) $X = \{x_1, \dots, x_n\}, n \in \mathbb{N}$, and let $K \leq F$ be a subgroup of finite index (respectively, $K \leq_o F$). Then for any $t \in X \cup X^{-1}$ not in $K$, there is a basis (respectively, profinite basis) $Y$ for $K$ such that:

\begin{enumerate}

\item If some $g \in K$ (respectively, $g \in K$ which is in the abstract subgroup generated by $X$) starts with $t$ then $length_Y(g) < length_X(g)$.

\item $X \cap K \subseteq Y$.

\end{enumerate}

\end{proposition}

\begin{proof}

For abstract free groups we have \cite[Proposition 17.5.6]{FJ}, so we shall treat only the profinite case. Let $\Phi$ be the free abstract group generated by $X$, and note that $[\Phi : (\Phi \cap K)] \leq [F : K] < \infty$ and that $t \notin (\Phi \cap K)$ so we can use the abstract version of our proposition to find a basis $Y$ for $\Phi \cap K$ such that $X \cap K \subseteq X \cap (K \cap \Phi) \subseteq Y$. Furthermore, for any $g \in (K \cap \Phi)$ the inequality $\text{length}_Y(g) < \text{length}_X(g)$ holds. To conclude, note that by \propref{AbstProfBasisProp}, $Y$ is a profinite basis for $K$. $\blacksquare$

\end{proof}

\propref{BasesProp} as well as several other results in the following section, has an appropriate analogue for the case of infinite rank. These analogues have more technical and complicated statements, and they are not necessary for proving our main assertions. Hence, in order to facilitate our exposition and to keep it as elementary as possible, we stick to the finitely generated case where appropriate.

We finish the background section with a simple application of a lemma of Gaschutz (see \cite[Lemma 17.7.2]{FJ}).

\begin{proposition} \label{GaschProp}

Let $F$ be a finitely generated free profinite group, and let $N \lhd_c F$. Then there exists a basis $X$ of $F$ such that: $$|X \cap N| \geq d(F) - d(F/N).$$

\end{proposition}

\begin{proof}

Let $\phi \colon F \to F/N$ be the quotient map, and set $n = d(F)$, $m = d(F/N)$. Pick some $h_1, \dots, h_m$ generating $F/N$, and apply \cite[Lemma 17.7.2]{FJ} to $\phi \colon F \to F/N$ and to the system $h_1, \dots, h_m, 1, \dots, 1$ of generators of $F/N$ consisting of $n$ elements. We get a generating set $X = \{x_1, \dots, x_n\}$ for $F$, such that $\phi(x_j) = 1$ for each $m < j \leq n$. By \cite[Lemma 3.3.5 (b)]{RZ}, $X$ is a basis for $F$, and $x_{m+1}, \dots, x_n \in (X \cap N)$ as required. $\blacksquare$

\end{proof}
   
\section{Combinatorics, Words, and \\* Embedding Problems}

In view of Iwasawa's criterion (see \cite[Corollary 24.8.3]{FJ}), the ability to properly solve embedding problems constitutes the technical heart of establishing the profinite freeness of a profinite group. Here, we give several combinatorial arguments that will aid us in acquiring this ability.

We present a basic claim, the proof of which is an application of the pigeonhole principle.

\begin{proposition} \label{SimpleSolProp}

Let $F$ be a finitely generated free profinite group with basis $X = \{x_1, \dots, x_r\}$, where $r \in \mathbb{N}$. Let $$\begin{tikzpicture}[scale=1.5]
\node (B) at (1,1) {$F$};
\node (C) at (0,0) {$A$};
\node (D) at (1,0) {$B$};
\path[->,font=\scriptsize,>=angle 90]
(B) edge node[right]{$\theta$} (D)
(C) edge node[above]{$\alpha$} (D);
\end{tikzpicture}$$be a finite embedding problem and put $k = |A|$. Suppose that $k \leq r$, and that $\lambda \colon \{x_{k+1}, \dots, x_r\} \rightarrow A$ satisfies $\alpha(\lambda(x_i)) = \theta(x_i)$ for each $k+1 \leq i \leq r$. Then there is a proper solution $\psi \colon F \rightarrow A$ satisfying:

\begin{enumerate}

\item For every $k+1 \leq i \leq r$, $\psi (x_i) = \lambda(x_i)$. 

\item For every $\{x_1, \dots, x_k\} \subseteq H \leq_c F$ with $\theta(H) = B$, it holds that $\psi(H) = A$.

\end{enumerate}

\end{proposition}  

\begin{proof}

By the pigeonhole principle, there is some $b \in B$ such that: $$|\theta^{-1}(b) \cap \{x_1, \dots, x_k\}| \geq \frac{k}{|B|} = \frac{|A|}{|B|} = |\text{Ker}(\alpha)| = |\alpha^{-1}(b)|.$$ Hence, we can and find a function $$\psi \colon \theta^{-1}(b) \cap \{x_1, \dots, x_k\} \rightarrow A$$ with image equal to $\alpha^{-1}(b)$. We extend $\psi$ to $X$ by taking $\psi(x_i)$ to be an arbitrary element of $\alpha^{-1}(\theta(x_i))$ for each $1 \leq i \leq k$ with $\theta(x_i) \neq b$, and $\psi(x_i) = \lambda(x_i)$ for $k+1 \leq i \leq r$. Since $X$ is a basis, $\psi$ extends (uniquely) to a continuous homomorphism $\psi \colon F \rightarrow A $, with $\alpha \psi = \theta$ on $X$, and thus on all of $F$. Now, let $\{x_1, \dots, x_k\} \subseteq H \leq_c F$, set $G = \psi(H)$ and note that $\alpha^{-1}(b) \subseteq G$ generates a subgroup that contains $\text{Ker}(\alpha)$, so: $$ G = G\text{Ker}(\alpha) =  \alpha^{-1}(\alpha(G)) = \alpha^{-1}(\theta(H)) =  \alpha^{-1}(B) = A$$ and we have a proper solution with $\psi(H) = A$ as required. $\blacksquare$

\end{proof}

In the following, we combine \propref{SimpleSolProp} and \propref{BasesProp} to derive some plausible condition on a closed subgroup of a finitely generated free profinite group, sufficient for properly solving a "bunch" of embedding problems.

\begin{proposition} \label{ContBaseProp}

Let $F$ be a finitely generated free profinite group, and let $H \leq_c F$ be a closed subgroup containing at least $k$ distinct elements of some basis of $F$, for some $k \in \mathbb{N}$. Then any finite embedding problem $$\begin{tikzpicture}[scale=1.5]
\node (B) at (1,1) {$H$};
\node (C) at (0,0) {$A$};
\node (D) at (1,0) {$B$};
\path[->,font=\scriptsize,>=angle 90]
(B) edge node[right]{$\theta$} (D)
(C) edge node[above]{$\alpha$} (D);
\end{tikzpicture}$$ with $|A| \leq k$ is properly solvable.

\end{proposition}

\begin{proof}

Consider such an embedding problem. Using \mbox{\cite[Lemma 1.2.5 (c)]{FJ}}, we extend $\theta$ to some $H \leq L \leq_o F$. By our assumption, $F$ has a basis $$Y = \{x_1, \dots, x_k,y_1, \dots , y_d\}$$ for some $d \in \mathbb{N}$, and $\{x_1, \dots, x_k \} \subseteq H$. Since $\{x_1, \dots , x_k\} \subseteq L$, we infer from \propref{BasesProp} (profinite version, (2)) that there is some $r \in \mathbb{N}$, and a basis $$X = \{x_1, \dots, x_k, \dots, x_r\}$$ of $L$ such that $k \leq r$. Since $\alpha$ is surjective, we can choose some $\lambda(x_i) \in \alpha^{-1}(\theta(x_i))$ for each $k+1 \leq i \leq r$. Applying \propref{SimpleSolProp} to the embedding problem $$\begin{tikzpicture}[scale=1.5]
\node (B) at (1,1) {$L$};
\node (C) at (0,0) {$A$};
\node (D) at (1,0) {$B$};
\path[->,font=\scriptsize,>=angle 90]
(B) edge node[right]{$\theta$} (D)
(C) edge node[above]{$\alpha$} (D);
\end{tikzpicture}$$ to $X$, and to $\lambda$, we get a proper solution $\psi \colon L \rightarrow A$ satisfying \mbox{$\psi(H) = A$} (since $\theta(H) = B$), so we have a proper solution for our initial embedding problem as required. $\blacksquare$

\end{proof}

Besides the role that the next lemma plays in the upcoming corollaries, it can be used to generalize a group theoretic lemma of Levi (see \cite[Lemma 17.5.10]{FJ}), and a theorem of Takahasi \mbox{(see \cite[Theorem 2.12]{MKS})} to arbitrary descending chains (by appealing to \cite[Proposition 17.5.6]{FJ} instead of \propref{BasesProp} which is used in the proof of our lemma). A more general form of our lemma, appears in \cite[Section 2.4, Exercise 36]{MKS}.

\begin{lemma} \label{AbstChangeVarLem}

Let $F$ be a free group on a finite set $X$, and let $\{U_i\}_{i \in \mathbb{N}}$ be a strictly descending chain of finite-index subgroups. Put $$L = \bigcap_{i \in \mathbb{N}} U_i$$ and suppose that $\text{rank}(L) \geq n$ for some $n \in \mathbb{N}$. Then there exists some $j \in \mathbb{N}$ and a basis $C$ of $U_j$ such that $|C \cap L| \geq n$.

\end{lemma}

\begin{proof}

Suppose that the conclusion does not hold, and set: $$R = \{(m,D) : m \in \mathbb{N}, \hspace{1mm} D \text{ is a basis of } U_m \}.$$ By the Nielsen-Schreier theorem, subgroups of $F$ possess bases, so $R \neq \emptyset$, and we can pick some $(j,A) \in R$ with $r = |A \cap L| < n$ the largest possible over all of $R$. Since $r < \rank(L)$, it follows that there is some $l \in L$ which is not generated by $A \cap L$. Define: $$S = \{ (x,i,E) : x \in L, \hspace{1mm} x \notin \langle A \cap L \rangle, \hspace{1mm} (A \cap L) \subseteq E, \hspace{1mm} (i,E) \in R  \}.$$ By the above, $(l,j,A) \in S$ so $S \neq \emptyset$, and we can choose some $(w,d,B) \in S$ which minimizes $\text{length}_B(w)$ over $S$. Let $$w = b_1^{\epsilon_1} \dotsi b_k^{\epsilon_k}$$ be the corresponding reduced presentation with $b_i \in B$, $\epsilon_i \in \{\pm 1\}$, where $1 \leq i \leq k$ and $k = \text{length}_B(w)$.

Suppose first that $b_1 \in L$. If $b_1 \in (A \cap L)$, then $(b_1^{-\epsilon_1}w,d,B) \in S$ and $\text{length}_B(b_1^{-\epsilon_1}w) = \text{length}_B(b_2^{\epsilon_2} \dotsi b_k^{\epsilon_k}) \leq k - 1$ which is a contradiction to minimality. On the other hand, if $b_1 \notin (A \cap L)$ then $A \cap L \subseteq B \cap L$ and $b_1 \in B \cap L$ imply that $|B \cap L| \geq r + 1$ which contradicts the choice of $r$.

Now assume that $b_1 \notin L$, so there is some $d < p \in \mathbb{N}$ such that $b_1 \notin U_p$. Applying \propref{BasesProp} to $b_1^{\epsilon_1} \notin U_p \leq U_d$ and the basis $B$, we conclude that there is a basis $Y$ for $U_p$ such that $$(A \cap L) \subseteq (B \cap L) \subseteq (B \cap U_p) \subseteq Y$$ and $\text{length}_Y(w) < \text{length}_B(w)$ since $w \in L \leq U_p$ starts with $b_1^{\epsilon_1}$. It is clear that $(w,p,Y) \in S$ so we have a contradiction to minimality. $\blacksquare$

\end{proof}

In the following sequence of corollaries, we establish the verbal subgroup theorem (see \thmref{FreeVerbalSubgroupThm}).

\begin{corollary} \label{ChangeVarCor}

Let $F$ be a free profinite group on a finite set $X$, and let $H \leq_c F$. Denote by $\Phi$ the abstract subgroup of $F$ generated by $X$, and suppose that $\rank(H \cap \Phi) \geq n$ for some $n \in \mathbb{N}$. Then there exists some $H \leq V \leq_o F$, and a basis $C$ of $V$ such that $|H \cap C| \geq n$.

\end{corollary}

\begin{proof}

If $H \leq_o F$, then we can just take $V = H$ and any basis of it, so we may assume that there is a strictly descending sequence $\{V_i\}_{i \in \mathbb{N}}$ of open subgroups of $F$ intersecting in $H$ (just combine \mbox{\cite[Proposition 2.5.1 (b)]{RZ}} with \cite[Proposition 2.1.4 (a)]{RZ}). Set $$U_i = V_i \cap \Phi$$ for every $i \in \mathbb{N}$, and note that by \cite[Proposition 3.2.2 (a),(d)]{RZ}, $[\Phi : U_i] = [F : V_i]$ for all $i \in \mathbb{N}$, so we have defined a strictly decreasing sequence of finite index subgroups of $\Phi$ intersecting in $\Phi \cap H$. By \lemref{AbstChangeVarLem}, there is some $j \in \mathbb{N}$, and a basis $C$ of $U_j$ such that $|H \cap C| \geq |(H \cap \Phi) \cap C| \geq n$. Finally, we infer that $C$ is a basis for $V_j$ from \propref{AbstProfBasisProp}. $\blacksquare$

\end{proof}

\begin{corollary} \label{RankCor}

Let $F$ be a free profinite group on a finite set $X$, and let $H \leq_c F$. Denote by $\Phi$ the abstract subgroup of $F$ generated by $X$, and suppose that $\rank(H \cap \Phi) \geq n$ for some $n \in \mathbb{N}$. Then any finite embedding problem $$\begin{tikzpicture}[scale=1.5]
\node (B) at (1,1) {$H$};
\node (C) at (0,0) {$A$};
\node (D) at (1,0) {$B$};
\path[->,font=\scriptsize,>=angle 90]
(B) edge node[right]{$\theta$} (D)
(C) edge node[above]{$\alpha$} (D);
\end{tikzpicture}$$ with $|A| \leq n$ is properly solvable.

\end{corollary}

\begin{proof}

This is immediate from \lemref{ChangeVarCor} and \propref{ContBaseProp}. $\blacksquare$

\end{proof}

\begin{corollary} \label{InfCor}

Let $F$ be a free profinite group on a finite set $X$, and let $H \leq_c F$. Denote by $\Phi$ the abstract subgroup of $F$ generated by $X$, and suppose that $\mathrm{rank}(H \cap \Phi) = \infty$. Then $H$ is a free profinite group of rank $\aleph_0$.

\end{corollary}

\begin{proof}

By \corref{RankCor}, every finite embedding problem for $H$ is properly solvable, so we invoke \cite[Corollary 24.8.3]{FJ}. $\blacksquare$

\end{proof}

Now, the verbal subgroup theorem (see \thmref{FreeVerbalSubgroupThm}) is a corollary.

\begin{corollary} \label{ContWordCor}

Let $F$ be a free profinite group on a finite set $X$, and let $H \leq_c F$ be a closed subgroup of infinite index in $F$. Denote by $\Phi$ the abstract subgroup of $F$ generated by $X$, and suppose that there is some $N \lhd_c F$ contained in $H$, such that $N \cap \Phi \neq \{1\}$. Then $H$ is a free profinite group of rank $\aleph_0$.

\end{corollary} 

\begin{proof}

Since $[F : H] = \infty$, there is some $H \leq U \leq_o F$ of arbitrarily high index (see \cite[Proposition 2.1.4 (d)]{RZ}). Hence, $$[\Phi : H \cap \Phi] \geq [\Phi : U \cap \Phi] = [F : U]$$ by \cite[Proposition 3.2.2 (a),(d)]{RZ}. Therefore, $[\Phi : H \cap \Phi] = \infty$, $N \cap \Phi \leq H \cap \Phi$, and $\{1\} \neq N \cap \Phi \lhd \Phi$, so by a theorem of Karrass and Solitar (see \cite{KS}), $\rank(H \cap \Phi) = \infty$, and we apply \corref{InfCor}. $\blacksquare$

\end{proof}

A proof of the free intermediate subgroup theorem is given (see \thmref{FreeInterSubThm}).

\begin{theorem} \label{RouteThm}

Let $F$ be a free profinite group of finite rank \mbox{exceeding $1$,} and let $H \leq_c F$ be a closed subgroup of infinite index in $F$. Then there exists a free profinite subgroup $H \leq L \leq_c F$ of rank $\aleph_0$.

\end{theorem}

\begin{proof}

We inductively construct an ascending sequence of closed subgroups $\{H_n\}_{n \in \mathbb{N}}$, and a strictly descending sequence of open subgroups $\{F_n\}_{n \in \mathbb{N}}$ subject to:

\begin{enumerate}

\item $H_0 = H$, and $F_0 = F$.

\item $H_n \leq F_n$, and $[F_n : H_n] = \infty$.

\item Any finite embedding problem $$\begin{tikzpicture}[scale=1.5]
\node (B) at (1,1) {$K$};
\node (C) at (0,0) {$A$};
\node (D) at (1,0) {$B$};
\path[->,font=\scriptsize,>=angle 90]
(B) edge node[right]{$\theta$} (D)
(C) edge node[above]{$\alpha$} (D);
\end{tikzpicture}$$with $|A| \leq n$, and infinite index subgroup $H_n \leq K \leq_c F_n$ is properly solvable.

\end{enumerate}

\underline{\textbf{Part A: Inductive construction}} 

Since our conditions are satisfied for $n = 0$, we pick some positive $n \in \mathbb{N}$, and assume that $F_{n-1}$ and $H_{n-1}$ have already been defined. We can pick some $H_{n-1} \leq K_0 \leq_c F_{n-1}$ with $[F_{n-1} : K_0] = \infty$ and an embedding problem $$\begin{tikzpicture}[scale=1.5]
\node (B) at (1,1) {$K_0$};
\node (C) at (0,0) {$A_0$};
\node (D) at (1,0) {$B_0$};
\path[->,font=\scriptsize,>=angle 90]
(B) edge node[right]{$\theta_0$} (D)
(C) edge node[above]{$\alpha_0$} (D);
\end{tikzpicture}$$ with $|A_0| \leq n$, which is not properly solvable, for if such a choice cannot be made, we set $H_n = H_{n-1}$ and take $F_n$ to be any proper open subgroup of $F_{n-1}$ containing $H_n$ (this is possible in view of (2)), thus finishing our inductive construction since (2) and (3) are satisfied. Extend $\theta_0$  to some $$K_0 \leq U \leq_o F_{n-1}$$ using \cite[Lemma 1.2.5]{FJ}, and note that since $[U : K_0] = \infty$ we may well assume that $m = \text{rank}(U) \geq 2n$, and that $U \neq F_{n-1}$ as it is possible to switch $U$ by some proper open subgroup, thus arbitrarily increasing the rank in view of \mbox{\cite[Theorem 3.6.2 (b)]{RZ}}. It follows from \mbox{\propref{GaschProp}}, and from the inequalities $$d(B_0) \leq d(A_0) \leq |A_0| \leq n$$ that there is a basis $\{u_1, \dots, u_m\}$ for $U$, such that \mbox{$\{u_{n+1}, \dots, u_{m}\} \subseteq \text{Ker}(\theta_0)$}. Set $$H_n = \langle K_0, u_{n+1}, \dots, u_{m} \rangle, \hspace{1mm} F_n = U.$$ It is clear that $H_{n-1} \leq H_n \leq_c F_n \lneq_o F_{n-1}$ so in order to finish our inductive construction, we need to verify conditions (2) and (3).

\underline{\textbf{Condition (2) holds:}}

Towards a contradiction, assume that $H_n \leq_o U$. By \mbox{\propref{BasesProp} (2)}, there is a basis $$C = \{c_1, \dots, c_n, u_{n+1}, \dots, u_{m} \}$$ for $H_n$. Note that the restriction of $\theta_0$ to $H_n$ is surjective since \mbox{$K_0 \leq H_n$.} Denote by $\lambda \colon \{u_{n+1}, \dots, u_m \} \rightarrow A_0$ the function which equals $1 \in A_0$ identically. It follows from \propref{SimpleSolProp} (applied to $H_n$, $m \geq 2n \geq |A|$, and $\lambda$), that the embedding problem $$\begin{tikzpicture}[scale=1.5]
\node (B) at (1,1) {$H_n$};
\node (C) at (0,0) {$A_0$};
\node (D) at (1,0) {$B_0$};
\path[->,font=\scriptsize,>=angle 90]
(B) edge node[right]{$\theta_0$} (D)
(C) edge node[above]{$\alpha_0$} (D);
\end{tikzpicture}$$ has a proper solution $\psi \colon H_n \rightarrow A_0$ such that $\{u_{n+1}, \dots, u_{m}\} \subseteq \text{Ker}(\psi)$. Now, $A = \psi(H_n) = \langle \psi(K_0), \psi(u_{n+1}), \dots, \psi(u_{m}) \rangle = \psi(K_0)$ so $\psi|_{K_0}$ contradicts our choice of $K_0$ and of the embedding problem.

\underline{\textbf{Condition (3) holds:}} 

Take a finite embedding problem $$\begin{tikzpicture}[scale=1.5]
\node (B) at (1,1) {$K$};
\node (C) at (0,0) {$A$};
\node (D) at (1,0) {$B$};
\path[->,font=\scriptsize,>=angle 90]
(B) edge node[right]{$\theta$} (D)
(C) edge node[above]{$\alpha$} (D);
\end{tikzpicture}$$ for an infinite index subgroup $H_n \leq K \leq_c F_n$, such that $|A| \leq n$. Clearly, $\{u_{n+1}, \dots, u_{m}\} \subseteq K$, and $m-n \geq 2n - n = n \geq |A|$ so a proper solution exists in view of \propref{ContBaseProp}. Thus, we have completed our induction.

\underline{\textbf{Part B: Constructing $L$}} 

Set $$L = \bigcap_{n \in \mathbb{N}} F_n$$ and note that $H_r \leq H_m \leq F_m \leq F_r$ for every $m,r \in \mathbb{N}$ with $m \geq r$. By omitting $H_m$ and taking the intersection over all $m \geq r$, we conclude that $H = H_0 \leq H_r \leq L \leq F_r$. Since $L$ is the intersection of a strictly descending series, for all $r \in \mathbb{N}$ we have $[F_r : L] = \infty$ so every finite embedding problem for $L$ is properly solvable in view of (3). Our assertion is now a consequence of \cite[Corollary 24.8.3]{FJ}. $\blacksquare$

\end{proof}

\section{An action of a compact group}

First, a little remark regarding our notation is in place. For an element $x \in \prod_i A_i$ we write $x_i$ for its $i$-th coordinate, and we identify $A_i$ with the subgroup of all elements $x$ with $x_j=1$, for all $j\neq i$. 

The following lemma will provide us with an abundance of finite normal subgroups, which are useful for proving profinite freeness. Intrinsically, it shows that in some special case, the action of a compact group disassembles into finite actions.

\begin{lemma} \label{CompactLem}

Let $G$ be a compact group acting continuously, by automorphisms, on a nontrivial product of nonabelian finite simple groups $$M = \prod_{i \in I}S_{i}$$ equipped with the product topology. Then there is a partition of $I$ into finite subsets $$I = \bigcup_{j \in J} I_{j}$$ such that for each $j \in J$ $$\prod_{i \in I_j}S_i$$ is $G$-invariant, and a power of a simple group.

\end{lemma}

\begin{proof}

Let $i \in I$, $g \in G$, and $1 \neq x \in S_i$. Clearly $S_i \lhd M$, so since $G$ acts by continuous automorphisms, $gS_i \lhd M$ is isomorphic to $S_i$. Since at least one of the projections of $gS_i$ defined by $I$ must be nontrivial, it follows from \cite[Lemma 8.2.4]{RZ} that there exists some $t \in I$ such that $gS_i = S_t$. Therefore, $G$ acts as a permutation group on $I$. Furthermore, from the same lemma, we deduce that:
$$G_{S_i} = \{h \in G : hS_i = S_i\} = \{h \in G : (hx)_i \neq 1\}.$$

Since $\{m \in M : m_i \neq 1\}$ is open, the continuity of the action of $G$ on $M$ guarantees that $$V = \{(a,m) \in G \times M : (am)_i \neq 1\}$$ is open as well. Since $(1,x) \in V$, the definition of the product topology provides us with an open neighborhood $U$ of $1$ in $G$ such that \mbox{$U \times \{x\} \subseteq V$}, which means that $U \subseteq G_{S_i}$. Clearly, $G_{S_i} = UG_{S_i}$ is an open subgroup of $G$. The cosets of $G_{S_i}$ form a disjoint open covering of the compact group $G$, so $[G : G_{S_i}]$ is finite. It follows that the $G$-orbit of $i$ is finite, so partitioning $I$ into its $G$-orbits gives us the result. $\blacksquare$

\end{proof}

In the following proposition, we achieve profinite freeness using the finite normal subgroups guaranteed to exist by the properties of the action of a compact group as proved in the previous lemma.

\begin{proposition} \label{PushingUpProp}

Let $F$ be a free profinite group of rank $m$, and set $m^* = \max\{m,\aleph_0\}$. Let $N \lhd_c F$, put $R = F/N$, and let $N \leq K \leq_{c} F $ be a closed subgroup of infinite index in $F$. Suppose that there is a closed normal subgroup $N \leq L \lhd_c F$ not contained in $K$, such that $L/N$ is a direct product of nonabelian finite simple groups. Then $K$ is a free profinite group of rank $m^*$.

\end{proposition}

\begin{proof}

It is apparent from our assumptions that $L/N \lhd_c R$, so $R$ acts by continuous automorphisms on $L/N$ via conjugation. By \mbox{\lemref{CompactLem}}, $L/N$ is a direct product of $R$-invaraint, thus $R$-normal, finite subgroups. Since $K/N$ does not contain $L/N$, it does not contain one of the aforementioned finite normal subgroups. That is, there is some closed subgroup of $L$, $N \lneq M \lhd_c F$ not contained in $K$, with $1 < [M : N] < \infty$. Therefore, $K$ is a proper open subgroup of $MK$ not containing $M \lhd_c F$, so $K$ is a free profinite group of rank $m^*$ by \cite[Proposition 1.3]{Ja06} and the remark following it. $\blacksquare$

\end{proof}

\section{Groups of finite length}

We recall the definition of a class of profinite groups defined in \cite{BFW}.

\begin{definition}

Let $G$ be a profinite group. We define the \textbf{\textit{generalized derived subgroup}} $D(G)$ of $G$ to be the intersection of all open normal subgroups $M \lhd_o G$ for which $G/M$ is either an abelian group or a nonabelian finite simple group. The \textbf{\textit{generalized derived series}} is defined inductively: $D_0(G) = G$ and $D_{i+1}(G) = D(D_{i}(G))$ for $i \geq 0$. Plainly, this is a strictly descending series of closed normal subgroups of $G$ (see \cite[Lemma 2.1]{BFW}). Moreover, we define the \textbf{\textit{abelian-simple length}} of $G$, denoted by $l(G)$, to be the smallest $r \in \mathbb{N}$ for which $D_{r}(G) = \{1\}$. In the absence of such a natural number, we set $l(G) = \infty$. 

\end{definition} \label{FSLDef}

The following claims give some basic properties related to abelian-simple length, and generalized derived subgroups.  

\begin{proposition} \label{FinLindexProp}

Let $G$ be a profinite group with $l(G) < \infty$, and let $U \leq_o G$. Then $l(U) < \infty$.

\end{proposition}

\begin{proof}

Set $N = U_G \lhd_o G$ (the intersection of all the $G$-conjugates of $U$). By \cite[Proposition 2.9, Lemma 2.2, Lemma 2.7 (2)]{BFW}: $$l(U) \leq l(U/N) + l(N) \leq \text{log}_2(|U/N|) + l(G) < \infty.$$ $\blacksquare$

\end{proof}

\begin{proposition} \label{abnonabProp}

Let $G$ be a profinite group, and set $M = G/D(G)$. Then there are characteristic subgroups $A,P \leq_c M$ with $A$ abelian and $P$ a product of nonabelian finite simple groups, such that $M = A \times P$.

\end{proposition}

\begin{proof}

Let $A$ be the intersection of all $N \lhd_o M$ with $M/N$ a nonabelian finite simple group, and let $P$ be the commutator of $M$. Clearly, $A$ and $P$ are characteristic in $M$, and $A \cap P = \{1\}$ (by definition of $D(G)$). Towards a contradiction, suppose that $AP \lneq M$, and pick some maximal proper $AP \leq L \lhd_o M$. By \cite[Lemma 18.3.11]{FJ}, $M/A$ is isomorphic to a product of nonabelian finite simple groups, and $M/L$ is a nonabelian finite simple group (since $A \leq L$ is maximal). On the other hand, $P \leq_c L$ so $M/L$ is abelian, leading to a contradiction. Therefore, $AP = M$ so $M = A \times P$, and $P \cong M/A$ so $P$ is isomorphic to a product of nonabelian finite simple groups. Finally, $A \cong M/P$ which is an abelian group. $\blacksquare$

\end{proof}

We shortly recall the terminology of twisted wreath products. Let $A$ and $G_0 \leq G$ be finite groups with a right action by group automorphisms of $G_0$ on $A$. Define: $$\text{Ind}_{G_0}^{G}(A) = \{f \colon G \rightarrow A \hspace{1mm} | \hspace{1mm} f(xy) = f(x)^y \hspace{1mm} \forall x \in G, y \in G_0 \}.$$ This is a group under pointwise multiplication, on which $G$ acts by group automorphisms from the right by $f^{\sigma}(\tau) = f(\sigma\tau)$,  for $\sigma, \tau \in G$. The \textbf{\textit{twisted wreath product}} is defined
to be the semidirect product $$A \hspace{1mm} \text{wr}_{G_0} \hspace{1mm} G = \text{Ind}_{G_0}^{G}(A) \rtimes G.$$

In the following theorem we utilize the twisted wreath product approach to prove the generalization of \cite[Theorem 5.7]{BFW} to arbitrary infinite rank, thus treating the case of infinite rank in \thmref{FSLThm}. Our proof is just a translation to group theory (using the Galois correspondence) of the field theoretic argument given in the proof of \cite[Theorem 3.2]{BFW}.

\begin{theorem} \label{FIThm}

Let $F$ be a free profinite group of infinite rank $r$, and let $N \lhd_c F$ be such that $F/N$ is of finite abelian-simple length. Then any $N \leq R \leq_c F$ is a free profinite group of rank $r$.

\end{theorem}

\begin{proof}

By the finite length assumption, there exists a minimal integer $m \geq -1$ for which $D_{m+1}(F/N) = \{1\}$. By the definition of $m$ and by \cite[Lemma 2.7 (1)]{BFW}, $l(F/D_m(F)) < l(F/N)$ so $D_m(F)R$ is a free profinite group of rank $r$ (by induction on $m$, with the base case $m = -1$, i.e $F = N$ being trivial). Hence, if $[D_m(F)R : R] < \infty$, then $R$ is a free profinite group of rank $r$ by \cite[Theorem 3.6.2]{RZ}. 

We may thus assume that $[D_m(F)R :R] = \infty$. In order to show that $R$ is a free profinite group, we apply \cite[Theorem 2.2]{Ha}. Let $R \leq K_\alpha \leq_o F$, and $K_\beta \leq_o F$. By the assumption on the index, we can choose some $L \lhd_o F$ contained in $ K_\alpha \cap K_\beta$, such that $$[D_m(F)RL : RL] > 2^m.$$ Set $G = F/L$, $K = RL$, $N_0 = N \cap L$, and $G_0 = K/L$. Let $A$ be a nontrivial finite group equipped with an action of $G_0$ by group automorphisms. By \cite[Theorem 2.2]{Ha}, we need to show that the embedding problem $$\begin{tikzpicture}[scale=1.5]
\node (B) at (1,1) {$F/N_0$};
\node (C) at (-0.5,0) {$A \hspace{1mm} \text{wr}_{G_0} \hspace{1mm} G$};
\node (D) at (1,0) {$G$};
\path[->,font=\scriptsize,>=angle 90]
(B) edge node[right]{$\theta$} (D)
(C) edge node[above]{$\alpha$} (D);
\end{tikzpicture}$$ does not have a proper solution.

Towards a contradiction, suppose that there is a proper solution $\lambda \colon F/N_0 \rightarrow A \hspace{1mm} \text{wr}_{G_0} \hspace{1mm} G$. Applying \cite[Lemma 2.7 (1)]{BFW} to $F \to F/L$ we see that $$[D_m(G)G_0 : G_0] = [D_m(G)RL/L : RL/L] = [D_m(F)RL : RL] > 2^m.$$ In light of that, \cite[Proposition 2.11]{BFW} tells us that $$I = D_{m+1}(A \hspace{1mm} \text{wr}_{G_0} \hspace{1mm} G) \cap \text{Ind}_{G_0}^{G}(A) \neq \{1\}.$$ Take a nontrivial $t \in I$. By \cite[Lemma 2.7 (1)]{BFW}, there is some $T \in D_{m+1}(F/N_0)$ such that $\lambda(T) = t$. From our definition of $m$ and \cite[Lemma 2.7 (1)]{BFW}, we infer that $D_{m+1}(F) \leq_c N$, so: $$ D_{m+1}(F/N_0) = D_{m+1}(F)N_0/N_0 \leq N/N_0$$ and thus $T \in N/N_0$. Since $t \in \text{Ind}_{G_0}^G(A)$, so $$\theta(T) = \alpha(\lambda(T)) = \alpha(t) = 1$$ which means that $T \in \text{Ker}(\theta) = L/N_0$. It follows that $$T \in N/N_0 \cap L/N_0 = \{1\}.$$ This is a contradiction to the choice of $t$ since $t = \lambda(T) = \lambda(1) = 1$. $\blacksquare$

\end{proof}

Now, we combine all of the tools developed in this work to establish \thmref{FSLThm} (in view of \thmref{FIThm}, we should only treat the finitely generated case).

\begin{theorem} \label{FFSLThm}

Let $F$ be a free profinite group of finite rank $m\geq2$. Let $N \lhd_c F$, set $R = F/N$, and suppose that $l(R) < \infty$. Then any $N \leq H \leq_c F$ of infinite index is a free profinite group of rank $\aleph_0$.

\end{theorem}

\begin{proof}

Let $\delta : F \rightarrow R$ be the quotient map, and let $H_0 = \delta(H)$. We may assume, without loss of generality, that $N = H_F$ (that is, there are no proper closed $F$-normal subgroups between $N$ and $H$) for otherwise we change $N$ accordingly (taking it to be the intersection of all the $F$-conjugates of $H$), and the length does not grow in view of \cite[Lemma 2.7 (1)]{BFW}. Since $l(R) < \infty$, there is a maximal $k \in \mathbb{N}$ for which $D_k(R)$ is nontrivial. By \propref{abnonabProp}, $D_k(R) = D_k(R)/D(D_k(R)) = A \times P$ where $P$ is a direct product of nonabelian finite simple groups, $A$ is an abelian group, and $A,P$ are characteristic in $D_k(R)$, and thus $A,P \lhd_c R$. If $P \neq \{1\}$, then $P \nleq H_0$ and we arrive at the desired conclusion by applying \propref{PushingUpProp}.

Therefore, we may assume that $P = \{1\}$ so $A = D_k(R) \neq \{1\}$ by definition of $k$. Since $A \lhd_c R$, we can intersect $A$ with a system coming from \cite[Theorem 2.1.3 (c)]{RZ}, to obtain a descending chain of open subgroups of $A$, $\{U_i\}_{i \in I}$ such that $U_i \lhd_c R$ for each $i \in I$ and $$\bigcap_{i \in I}U_i = \{1\}.$$ Since $A \nleq H_0$, we infer from \cite[Proposition 2.1.4 (a)]{RZ} that there is some $j \in I$ such that $A \nleq H_0U_j$. We distinguish between two cases:

\underline{\textbf{Case I :}} $[R :  H_0A] < \infty.$

By \thmref{RouteThm}, we can pick a free profinite subgroup \mbox{$H \leq M \leq_c F$} of rank $\aleph_0$, and set $M_0 = \delta(M)$. By \cite[Proposition 2.9]{BFW}, we have $$l(M_0) \leq l(M_0/M_0 \cap A) + l(M_0 \cap A) \leq l(M_0A/A) + 1$$ since $M_0 \cap A$ is abelian. By \cite[Lemma 2.7 (1)]{RZ}, $l(R/A) \leq l(R) < \infty$ and $[R/A : M_0A/A] \leq [R : H_0A] < \infty$ by our assumption, so we use \propref{FinLindexProp} to conclude that $l(M_0) \leq l(M_0A/A) + 1 < \infty$. The profinite freeness of $N \leq_c H \leq_c M$ is a consequence of \thmref{FIThm}.

\underline{\textbf{Case II :}} $[R :  H_0A] = \infty.$

Set $U = \delta^{-1}(U_j)$, $B = \delta^{-1}(A)$ and note that $B \nleq HU$. By \mbox{\cite[Proposition 2.1.4 (d)]{RZ}}, there is some $HU \leq V \leq_o F$ not containing $B$, with $$[F : V] \geq (2|B/U| + 4)/(d(F) - 1)$$ so that $d(V) \geq 2|B/U| + 5$ (see \cite[Theorem 3.6.2 (b)]{RZ}). \newline Using \cite[Theorem 3.6.2 (b), Corollary 3.6.3]{RZ} and the fact that \mbox{$V \lneq VB$}, we see that $$d(V/U) \leq d(V/(V \cap B)) + d((V \cap B)/U) \leq d(VB/B) + |(V \cap B)/U| \leq $$ $$\leq (d(F/B) - 1)[F/B : VB/B] + 1 + |B/U| \leq (d(F) - 1)[F : VB] + |B/U| + 1 = $$ $$ = (d(V) - 1)\frac{[F:VB]}{[F:V]} + |B/U| + 1 \leq (d(V) - 1)/2 + |B/U| + 1 \leq d(V) - 2$$ where the last inequality follows at once from our bound on $d(V)$. That is, we have shown that $d(V) - d(V/U) \geq 2$. From \mbox{\propref{GaschProp}}, we conclude that there is a basis $\{ v_1, \dots, v_n \}$ of $V$ with $v_1, v_2 \in U$. $U/N = U_j$ is abelian, so the commutator $[v_1,v_2] \in  N$, and this allows us to use \corref{ContWordCor} to conclude that $N \leq_c H \leq_c V$ is a free profinite group of rank $\aleph_0$. $\blacksquare$

\end{proof}

\section{Applications}

Here, several applications of our main result are given. In what follows, we formulate and prove the group theoretic analogue of \mbox{\cite[Theorem 5.4]{BFW}}.

\begin{corollary} \label{IntersectionBoundedIndexCor}

Let $F$ be a free profinite group. Let $d \in \mathbb{N}$, and  let $\mathcal{D}$ be the family of all open subgroups of $F$ of index not exceeding $d$. Set $K = \bigcap \mathcal{D}$, and let $K \leq M \leq_c F$. It follows that $M$ is a free profinite group.

\end{corollary}

\begin{proof}

Denote by $N \lhd_c F$ the intersection of all conjugates of the subgroups in $\mathcal{D}$, and pick some $L \in D$. Since $$[F/N : L/N] = [F : L] \leq d$$ we conclude that for the intersection of the conjugates of $L$ we have: $$[F/N : L_F/N] = [F : L_F] \leq d!$$ so by \cite[Lemma 2.2]{BFW}, $l((F/N)/(L_F/N)) \leq \log_2(d!)$. By our definition of $N$, $$ \bigcap_{R \in D} R_F/N = \{1\}$$ in $F/N$, so by \cite[Proposition 2.8]{BFW}, $F/N$ is of finite abelian-simple length, and our assertion is a consequence of \thmref{FSLThm} in view of the fact that $N \leq K \leq M \leq F$. $\blacksquare$

\end{proof}

Now, we prove the analogue of \cite[Theorem 5.3]{BFW}, thus generalizing \cite[Corollary 8.9.3]{RZ} for the formation of all finite groups.

\begin{corollary} \label{SolveCor}

Let $F$ be a nonabelian free profinite group, and let $N \lhd_c F$ be a closed normal subgroup of $F$ such that $F/N$ is solvable. Then any $N \leq K \leq_c F$ is a free profinite group. 

\end{corollary}

\begin{proof}

The generalized derived series of $F/N$ coincides with its derived series (formed by taking successive commutators) which terminates (at $\{1\}$) after a finite number of steps. Hence, $l(F/N) < \infty$ and we arrive at our result by applying \thmref{FSLThm}. $\blacksquare$

\end{proof}

Fix some $n \in \mathbb{N}$. By a finite $n$-dimensional representation of a profinite group we mean a continuous homomorphism to $\mathrm{GL}_n(K)$, for some field $K$ (equipped with the discrete topology). We give the group theoretic analogue of \cite[Theorem 5.1]{BFW}.

\begin{corollary} \label{FiniteRepsCor}

Let $F$ be a nonabelian free profinite group, and $n \in \mathbb{N}$. Denote
by $\Omega$ the family of all finite $n$-dimensional representations of $F$, and set $$N = \bigcap_{\rho \in \Omega} \mathrm{Ker}(\rho).$$ Then any $N \leq H \leq_c F$ is a free profinite group.

\end{corollary}

\begin{proof}

The argument given in the proof of \cite[Theorem 5.1]{BFW} is essentially group theoretic, showing that $l(F/N) < \infty$ so we may apply \thmref{FSLThm}. $\blacksquare$

\end{proof}

Next we prove the group theoretic analogue of \mbox{\cite[Theorem 1.2]{BFW}}, but let us first make a remark regarding \mbox{\cite[Lemma 4.3]{BFW}}. For this, recall that the subgroup rank of a topological group $G$ is defined to be: $$ r(G) = \sup \hspace{1mm} \{d(H) : H \leq_c G \}.$$ \cite[Lemma 4.3]{BFW} invokes \cite[Proposition 22.14.4]{FJ} which gives a uniform bound on the number of generators of all closed subgroups of a matrix group over the $p$-adic integers. Therefore, \mbox{\cite[Lemma 4.3]{BFW}} gives us in fact a little more: The subgroup rank of every compact subgroup of a matrix group over a finite extension of a $p$-adic field is finite. 

\begin{theorem} \label{AdicRepsThm}

Let $F$ be a nonabelian free profinite group, fix some $n \in \mathbb{N}$, and for each prime number $l$, let $\sigma_l \colon F \rightarrow \mathrm{GL}_n(\overline{\mathbb{Q}_{l}})$ be a continuous representation. Set $$N = \bigcap_{l} \mathrm{Ker}(\sigma_l).$$ Then any infinite index subgroup $N \leq H \leq_c F$ is a free profinite group of rank $\max\{\mathrm{rank}(F), \aleph_0\}$.

\end{theorem}

\begin{proof}

Let $l$ be a prime number, and set $N_l = \mathrm{Ker}(\sigma_l)$. Using \newline \cite[Corollary 4.6]{BFW}, and \cite[Lemmas 4.1-4.3]{BFW}, it is shown in the third and fourth paragraphs of the proof of \cite[Theorem 1.2]{BFW} that we can pick some $N_l \leq K_l \lhd_c F$ such that:

\begin{enumerate}

\item The abelian-simple length of $F/K_l$ is at most $m = m(n)$.

\item The group $K_l/N_l$ is pro-$l$.

\item The subgroup rank of $F/N_l$ is finite.

\end{enumerate} with the last condition coming from our remark (preceding the theorem) strengthening \cite[Lemma 4.3]{BFW}. Define the following normal subgroups of $F$ containing $N\colon$ $$K = \bigcap_{l} K_l, \hspace{1mm} K^p = K \cap N_p, \hspace{1mm} A_l = \bigcap_{p \neq l} N_p, \hspace{1mm} G_l = K \cap A_l, \hspace{1mm} G^q = \prod_{l \neq q} G_l $$ and note that if $l \neq q$, then $A_l \leq N_q$ so $G_l \leq K \cap N_q$, and we infer that $G^q \leq K \cap N_q = K^q$. For any two distinct primes $r,s$ it is easy to see that $G_r \cap G_s = N \leq H$, so if $G_r, G_s \nleq H$ we may invoke the diamond theorem to conclude that $H$ is a free profinite group.

Thus, we may assume that there is a prime number $q$, such that $G_l \leq H$ for any $l \neq q$ (or in other words, $G^q \leq H$). Note that for any prime $l$ we have an injection $$ K/G_l \rightarrow \prod_{p \neq l} K/(K \cap N_p) \cong \prod_{p \neq l} KN_p/N_p $$ into a product of groups of orders not divisible by $l$ since $$N_p \leq KN_p \leq K_p$$ and $K_p/N_p$ is pro-$p$. As a result, $[K : G_l]$ is not divisible by $l$, so since $[K: G^q]$ divides $[K : G_l]$ (apply \cite[Proposition 2.3.2 (d)]{RZ} to $G_l \leq_c G^q \leq_c K$) for every $l \neq q$, it follows that $[K : G^q]$ is not divisible by any prime other than $q$. Since $G^q \leq K^q \leq K$, $[K^q : G^q]$ and $[K : K^q]$ are powers of $q$ as well. It is therefore apparent that $K = G_qK^q$ as $[K : G_qK^q]$ divides both $[K : G_q]$ and $[K : K^q]$ which are coprime.

It is clear that $N = G_q \cap K^q$, so $K/N \cong G_q/N \times K^q/N$, and as a result $K^q/N \cong K/G_q$, showing us that $q$ does not divide $[K^q : N]$. Consequently, $[K^q : G^q]$ is not divisible by $q$. We conclude that $$K^q = G^q$$ since we have shown that there is no prime dividing $[K^q : G^q]$.   

We have seen that $K \cap N_q = K^q = G^q \leq H$, so by the diamond theorem, we may assume that $H$ contains either $K$ or $N_q$. If $K \leq H$, then by \cite[Proposition 2.8]{BFW}, $l(F/K) = \sup_{p}l(F/K_p) \leq m$ so we invoke either \thmref{FFSLThm} or \thmref{FIThm}. On the other hand, if $N_q \leq H$ we may appeal to the next lemma (\lemref{FinSubRankLem}) since the subgroup rank of $F/N_q$ is finite. $\blacksquare$

\end{proof}

In the following lemma we complete the proof of \thmref{AdicRepsThm}. The proof uses the notions of semi-free and projective profinite groups defined in \cite{BHH} and in \cite[Section 7.6]{RZ} respectively.

\begin{lemma} \label{FinSubRankLem}

Let $F$ be a free profinite group of rank $m \geq 2$, and let $N \lhd_c F$ such that the subgroup rank of $F/N$ is finite. Then every infinite index subgroup $N \leq H \leq_c F$ is a free profinite group of rank $\max\{m,\aleph_0\}$.

\end{lemma}

\begin{proof}

First, assume that $m$ is infinite. Since the subgroup rank of $F/N$ is finite, it follows that $F/N$ is finitely generated and so is every continuous image of it, so in particular, $F/H_F$ is finitely generated. Hence, by \cite[Main Theorem II]{BHH}, $H$ is a semi-free profinite group of rank $m$. From \cite[Lemma 7.6.3 (b)]{RZ} (applied to the variety of all finite groups), we infer that $H$ is a projective group, and thus a free profinite group of rank $m$ in view of \cite[Theorem 3.6]{BHH}.

Now suppose that $m < \infty$. Since $[F : H] = \infty$, there are subgroups $H \leq V \leq_o F$ of arbitrarily high index (see \cite[Proposition 2.3.2 (b)]{RZ}). Since the subgroup rank of $F/N$ is finite, there is some $H \leq U \leq_o F$ with $d(U) > d(U/N)$ (since $m \geq 2$, the number of generators of open subgroups grows with the index as shown in \cite[Theorem 3.6.2 (b)]{RZ}). The profinite freeness of $H$ is now a consequence of \mbox{\cite[Lemma 1.2]{Ja06}}. $\blacksquare$ 

\end{proof}

At last, we give some applications of arithmetic flavor. Let us begin by fixing notation, and recalling some terminology. 

\begin{definition} \label{Hildef}

Let $L$ be a field, and denote by $\overline{L}$ its algebraic closure. We say that:

\begin{itemize}

\item $L$ is PAC (Pseudo Algebraically Closed) if for every \mbox{$f(X,Y) \in L[X,Y]$} which is irreducible in $\overline{L}[X,Y]$, there are $x_0,y_0 \in L$ with \mbox{$f(x_0,y_0) = 0$}.

\item $L$ is Hilbertian if for every irreducible polynomial \mbox{$f(X,Y) \in L[X,Y]$} separable in $Y$, there is some $x_0 \in L$ such that \mbox{$f(x_0,Y) \in L[Y]$} is irreducible.

\item $L$ is $\omega$-free if every finite embedding problem for $G_L$ is properly solvable.

\end{itemize}

\end{definition}

For an abelian variety $A/L$ ($A$ is defined over $L$), we have the following notation:

\begin{itemize}

\item For $k \in \mathbb{N}$, $A[k]$ denotes the set of points of $A$ whose order divides $k$, and $A_{\mathrm{tor}} = \bigcup_{k \in \mathbb{N}} A[k]$.

\item For a prime number $l$ set: $$A[l^{\infty}] = \bigcup_{n = 1}^{\infty} A[l^n], \hspace{1mm} T_l(A) = \varprojlim_{n \in \mathbb{N}} A[l^n].$$

\item For a subset $S \subseteq A$, denote by $L(S)$ the subfield of $\overline{L}$ generated over $L$ by the coordinates of points in $S$.

\end{itemize}  

Fix some integer $m \geq 2$. By \cite[Corollary 23.1.2]{FJ}, there exists a perfect PAC field $L$ whose absolute Galois group $G_L$ is a free profinite group of rank $m$, as this group is projective. Furthermore, for almost every $\sigma_1, \dots, \sigma_m \in \mathrm{Gal}(\mathbb{Q})^m$ (with respect to the Haar measure on $\mathrm{Gal}(\mathbb{Q})^m$) the subfield of $\overline{\mathbb{Q}}$ fixed by $\sigma_1, \dots, \sigma_m$ is a perfect PAC field with absolute Galois group free profinite of rank $m$ (see \mbox{\cite[Chapter 20]{FJ}).} The following corollary establishes \thmref{AlmostThm}.

\begin{corollary} \label{AbVarCor}

Let $L$ be a perfect PAC field with a nonabelian finitely generated free profinite absolute Galois group. Let $A/L$ be an abelian variety, and let $K/L$ be an infinite extension contained in $L(A_{\mathrm{tor}})$. Then $K$ is Hilbertian. 

\end{corollary}

\begin{proof}

We may assume that $d = \text{dim}(A)$ is positive, for otherwise there is nothing to prove. For each prime number $l$ let $$\rho_l \colon G_L \to \text{GL}_{2d}(\mathbb{Q}_l)$$ be the Galois representation associated with the action of the absolute Galois group $G_L$ on the Tate module $T_l(A) \cong \mathbb{Z}_l^{2d}.$ Applying a basic property of inverse limits (\cite[Proposition 1.1.10]{RZ}) to $T_l(A)$, one is able to deduce that $\text{Ker}(\rho_l) = G_{L(A[l^\infty])}$, and thus $$\overline{L}^{\text{Ker}(\rho_l)} = L(A[l^{\infty}]).$$ Since any torsion abelian group is the sum of its $l$-subgroups, we infer that $$L(A_{\text{tor}}) = L(\bigoplus_l A[l^\infty]) = L(\bigcup_l A[l^\infty]) = $$ $$ = \prod_l L(A[l^\infty]) = \prod_l \overline{L}^{\text{Ker}(\rho_l)} = \overline{L}^{\bigcap_l \text{Ker}(\rho_l)}$$ where the second equality stems from the fact that the addition of a finite set of points from $A$ is given by rational functions (defined over $L$) of their coordinates.   

It is clear from the properties of the Galois correspondence that the absolute Galois group $G_K$ is a closed subgroup of infinite index in $G_L$, containing $\cap_l \text{Ker}(\rho_l)$. By \thmref{AdicRepsThm}, $G_K$ is a free profinite group of rank $\aleph_0$, and $K$ is PAC in view of \cite[Corollary 11.2.5]{FJ}. Hence, by \cite[Corollary 3.5.10]{RZ}, $K$ is $\omega$-free. We arrive at our result by applying \cite[Corollary 27.3.3]{FJ}. $\blacksquare$

\end{proof}

Our final application is a consequence of the intermediate subgroup theorem.

\begin{corollary} \label{IntSubCor}

Let $L$ be a perfect PAC field with a nonabelian finitely generated free profinite absolute Galois group. Let $K/L$ be an infinite extension. Then there exists a Hilbertian field $L \leq M \leq K$.

\end{corollary}

\begin{proof}

Clearly, $G_K$ is a closed subgroup of infinite index in $G_L$ so by \thmref{RouteThm}, there is a free profinite subgroup $G_K \leq H \leq_c G_L$ of rank $\aleph_0$. Set $M = \overline{L}^H$, and note that $M$ is $\omega$-free by \mbox{\cite[Corollary 3.5.10]{RZ}}, and PAC by \cite[Corollary 11.2.5]{FJ}. The Hilbertianity of $M$ follows from \cite[Corollary 27.3.3]{FJ}. $\blacksquare$

\end{proof}


\begin{thebibliography}{60}

\bibitem[BF]{BF} L. Bary-Soroker and A. Fehm, \emph{Random Galois extensions of Hilbertian fields}, J. The'or. Nombres Bordeaux, 24 (1), 31-42 (2013).

\bibitem[BFW]{BFW}  L. Bary-Soroker, A. Fehm, and G. Wiese, \emph{Hilbertian fields and Galois representations}, to appear in J. Reine Angew. Math.

\bibitem[BHH]{BHH} L. Bary-Soroker, D. Haran, and D. Harbater, \emph{Permanence criteria for semi-free profinite groups}, Math. Ann., 348 (3), 539-563 (2010).

\bibitem[BP]{BP} L. Bary-Soroker and E. Paran, \emph{Fully Hilbertian fields}, Israel J. Math., 194 (2), 507-538 (2013). 

\bibitem[BS]{BS} Lior Bary-Soroker, \emph{Diamond theorem for a finitely generated free profinite group}. Mathematische Annalen, 336(4):949–961, 2006.

\bibitem[DHLS]{DHLS} R. Downey, D. Hirschfeldt, S. Lempp, and R. Solomon, \emph{Reverse Mathematics of the Nielsen-Schreier Theorem}, Proceedings of International Conferences on Mathematical Logic, Novosibirsk State University Press, 2002, pp. 59 - 71. 

\bibitem[FJ]{FJ} M. Fried, M. Jarden, \emph{Field Arithmetic}, Ergebnisse der Mathematik
III 11. Springer, 2008. 3rd edition, revised by M. Jarden.

\bibitem[FJP]{FJP} A. Fehm, M. Jarden, and S. Petersen, \emph{Kuykian Fields}, Forum Mathematicum Volume 24, Issue 5, Pages 1013–1022, ISSN (Online) 1435-5337, ISSN (Print) 0933-7741.

\bibitem[FP]{FP} A. Fehm and S. Petersen. \emph{Hilbertianity of division fields of commutative algebraic groups}, Israel Journal of Mathematics 195(1), 2013. 

\bibitem[Ha99a]{Ha} D. Haran, \emph{Free subgroups of free profinite groups},
Journal of Group Theory 2, 307–317 (1999).

\bibitem[Ha99b]{Ha99} D. Haran, Hilbertian Fields under Separable Algebraic Extensions,
Inventiones mathematicae 137, 85–112 (1999). 

\bibitem[Hig]{Hig} P.J. Higgins, \emph{Notes on categories and groupoids}, Mathematical Studies, Volume 32. Van Nostrand Reinhold Co. London (1971); Reprints in Theory and Applications of Categories, No. 7 (2005) pp 1--195.

\bibitem[How]{How} P. Howard,  \emph{Subgroups of a free group and the axiom of choice}, The Journal of Symbolic Logic (1985), 50 (2): 458–467

\bibitem[Ja80]{Ja80} M. Jarden, \emph{Normal automorphisms of free profinite groups}, Journal of Algebra 62 (1980), 118--123.

\bibitem[Ja97]{Ja97} M. Jarden, \emph{Large normal extensions of Hilbertian fields}, Mathematische Zeitschrift 224 (1997), 555-565.

\bibitem[Ja06]{Ja06} M. Jarden, \emph{A Karrass-Solitar theorem for profinite groups}, Journal of Group theory 9 (2006), 139--146.

\bibitem[Ja10]{Ja10} M. Jarden, \emph{Diamonds in torsion of Abelian varieties}, Journal of the Institute of Mathematics Jussieu 9 (2010), 477--480.  
 
\bibitem[JL92]{JL92} M. Jarden, and A. Lubotzky, \emph{Hilbertian fields and free profinite groups}, Journal of the London Mathematical Society (2) 46 (1992), 205--227.

\bibitem[JL99]{JL99} M. Jarden, and A. Lubotzky, \emph{Random normal subgroups of free profinite groups}, Journal of Group Theory 2 (1999), 213--224.

\bibitem[KS]{KS} A.Karrass, and D.Solitar, \emph{Note on a theorem of Schreier}, Proc. Amer. Math. Soc. (8) (1957), 696-697.

\bibitem[MKS]{MKS} W.Magnus, A.Karrass, and D.Solitar, \emph{Combinatorial Group Theory}, Dover Publications, INC. New York, 1976.

\bibitem[Rot]{Rot} Rotman, (1995), \emph{An Introduction to the Theory of Groups}, Graduate Texts in Mathematics 148 (4th ed.), Springer-Verlag, ISBN 978-0-387-94285-8.

\bibitem[RS]{RS} Ribes L, and Steinberg B, \emph{A Wreath Product Approach to Classical Subgroup Theorems}, Enseign. Math. 56 (2010), 49-72.

\bibitem[RZ]{RZ} L. Ribes, and P. Zalesskii, \emph{Profinite Groups}, Ergebnisse der Mathematik und ihrer Grenzgebiete. 3. Folge / A Series of Modern Surveys in Mathematics 40, DOI 10.1007/978-3-642-01642-4, \textcopyright \space Springer-Verlag Berlin Heidelberg 2010.

\bibitem[SW]{SW} Stillwell, John (1993), \emph{Classical Topology and Combinatorial Group Theory}, Graduate Texts in Mathematics 72 (2nd ed.), Springer-Verlag.

\bibitem[Wil]{Wil} J.Wilson, \emph{Profinite Groups}, Clarendon Press, Oxford, 1998.

\end{thebibliography}
\end{document}